\documentclass[a4paper,12pt]{amsart}

\usepackage[english]{babel}
\usepackage{amsmath,amssymb,latexsym}

\usepackage[latin1]{inputenc}

\usepackage{graphicx}
\usepackage{psfrag}

\numberwithin{figure}{section}

\theoremstyle{plain}
\newtheorem{definition}{Definition}[section]

\newtheorem{theorem}[definition]{Theorem}

\newtheorem{lemma}[definition]{Lemma}

\theoremstyle{definition}
\newtheorem{remark}[definition]{Remark}
\newtheorem{example}[definition]{Example}

\newcommand{\Ch}{{\mathit{Ch}}}
\renewcommand{\S}{{\mathcal{S}}}

\title[Connectivity of chamber graphs]{Connectivity of chamber graphs \\ of buildings
and related complexes}
\author{Anders Bj\"orner}
\address{Royal Institute of Technology, Department of Mathematics,
  S-100 44 Stockholm, Sweden}
\email{bjorner@math.kth.se}
\author{Kathrin Vorwerk}
\address{Royal Institute of Technology, Department of Mathematics,
  S-100 44 Stockholm, Sweden}
\email{vorwerk@math.kth.se}
\thanks{Research supported by the Knut and Alice Wallenberg Foundation,
grant KAW.2005.0098}

\begin{document}

\begin{abstract} 
Let $\Delta$ be a finite building (or, more generally, a thick spherical and locally finite building).
The chamber graph $G(\Delta)$, whose edges are the pairs of adjacent chambers in $\Delta$,
is known to be $q$-regular for a certain number $q=q(\Delta)$. Our main result is that
$G(\Delta)$ is $q$-connected in the sense of graph theory.

Similar results are proved for the chamber graphs of 
Coxeter complexes and for order complexes of geometric lattices.

\end{abstract}

\maketitle

\section{Introduction}

Buildings were introduced by Tits \cite{Tits1974} for the purpose of creating a unified class
of geometric objects upon which  groups of Lie type act, and from which such
groups arise  as automorphism groups. This highly successful project has led to
a rich and elaborate theory, interweaving
group theory, geometry and combinatorics, see
\cite{AbramenkoBrown2008}, 
\cite{Ronan1989},
\cite{Scharlau1995},
\cite{Tits1974}.

From a purely combinatorial point of view, buildings can be defined and characterized in two ways.
First, they are  highly symmetric {\em simplicial complexes}. 
They arise by gluing together Coxeter complexes in a very symmetric way and can be interpreted as $q$-analogues of Coxeter complexes. The maximal simplices are called {\em chambers} and the embedded
Coxeter complexes are called {\em apartments}. This is the original 
point of view of Tits \cite{Tits1974}. 

Second, focussing on the chambers and their adjacency relation as the primitive objects of the theory,
buildings  can be characterized as a class of  {\em chamber systems}. This means that
one looks at the structure of the chamber graph, whose edges are the pairs of adjacent chambers,
embellished by certain labelling of these edges.
The chamber system point of view was introduced by Tits in later work and is exposited
e.g. in \cite{Ronan1989}.

The axioms for the system of apartments indicate that a building is very tightly held together.
It is therefore reasonable to expect also a high degree of connectivity of its chamber graph,
as measured by the number of pairwise disjoint paths (or, galleries) that connect any pair of chambers.
If  at least $q$ pairwise disjoint paths 
connect any pair of chambers, then the chamber graph
is said to be {\em $q$-connected}. The maximal such $q$ is the degree of connectivity of the graph.

The main result of this paper is that the chamber graph $G(\Delta)$  of
 a finite (or, more generally, thick spherical and locally finite) building $\Delta$
  is $q(\Delta)$-connected in the sense of graph theory. Here $q(\Delta)$ denotes
  the number of chambers adjacent to any given chamber of $\Delta$. Since more than $q(\Delta)$ 
  independent paths cannot leave a chamber,   it follows that the
  result is sharp, meaning that $q(\Delta)$  is the exact degree of connectivity of the chamber graph.
   
 Coxeter complexes are closely related to buildings. They appear as apartments in
 buildings as well as in many other contexts. 
  We show for a large class of $(d-1)$-dimensional Coxeter complexes
  that their chamber graph is $d$-connected. This class includes, for example,
the complexes of the classical affine Coxeter groups. Our method is constructive and relies strongly on the group structure behind the Coxeter complexes.

The buildings and the Coxeter complexes of type A are, respectively, the order complexes of
 subspace lattices of finite-dimensional vector spaces over some field, and the Boolean
lattices of subsets of some  finite set. Both these types of lattices are examples of geometric lattices.
In the last section we extend the study to chamber graphs of order complexes for  general 
geometric lattices.
 Again, we can prove a lower bound on the connectivity of the graph of those complexes. However, in this case the graphs need not be regular anymore and the bound needs not be sharp.

\section{Preliminaries}

\subsection{Graph theory}

We establish the graph-theoretic notions that are relevant for this paper. We follow the notation of \cite{Diestel2005} and refer to it for further details. In what follows, $G$ denotes a simple graph with vertex set $V(G)$ and edge set $E(G)$. The graph $G$ may be infinite, but we assume that the number of
edges incident to a vertex is always finite.

For $A,B \subseteq V(G)$, an {\em $A-B$ path} is a path starting at a vertex in $A$ and ending at a vertex in $B$ such that no interior vertex of the path is in $A \cup B$. If $A = \{a\}$ and $B = \{b\}$, then we call such a path  an $a-b$ path. Two $A-B$ paths are {\em disjoint} if the sets of their interior vertices are disjoint. The distance $d_G(u,v)$ between two vertices $u$ and $v$ of $G$ is the minimal length of an $u-v$ path.

A graph $G$ is called {\em $k$-connected} if $|V(G)| > k$ and $G$ remains connected after removing fewer than $k$ vertices and all incident edges. 
It is clear that in a $k$-connected graph every vertex is incident to at least $k$ edges. A well-known theorem by Menger \cite[Theorem 3.3.6]{Diestel2005}, valid also for infinite graphs, states that a graph is $k$-connected if and only if it contains $k$ disjoint $u-v$ paths for any two vertices $u,v \in V(G)$. This fact can be strengthened as follows.

\begin{lemma}[Liu's criterion, \cite{Liu1990}]
\label{liu}
Let $G$ be a connected graph and $|V(G)| > k$. If for any two
vertices $u$ and $v$ of $G$ with distance $d_G(u, v) = 2$ there are $k$ disjoint $u-v$
paths in $G$, then $G$ is $k$-connected.
\end{lemma}

For the proof, Liu in \cite{Liu1990} refers to another paper to which we have not had access, so we supply a proof.

\begin{proof}[Proof of Lemma \ref{liu}]
Assume that $G$ is not $k$-connected. By definition, there exists some $S \subset V(G)$ with $|S| < k$ such that $G - S$ is disconnected.
Choose $S$ minimal with respect to inclusion among all such sets. Then, there are two vertices $u$ and $v$ that lie in different components of $G - S$ and a $u-v$ path $P = (u = v_0, v_1, \ldots, v_n = v)$ that contains exactly one element of $S$. If $P \cap S = \{ v_i \}$ then $v_{i-1}$ and $v_{i+1}$ are also in different components of $G - S$ and thus there are at most $|S| \leq k-1$ disjoint $v_{i-1} - v_{i+1}$ paths in $G$. But $d_G(v_{i-1},v_{i+1}) = 2$, and we are done.
\end{proof}

A graph $G$ is called {\em $k$-regular} if every vertex $v \in V$ is contained in exactly $k$ edges. A $k$-regular graph is obviously at most $k$-connected.

\subsection{Chamber graphs of simplicial complexes}\label{ch.graphs}

Let $\Delta$ be a pure $d$-dimensional simplicial complex. The $d$-dimensional faces of $\Delta$ are called {\em chambers} and the $(d-1)$-dimensional faces are called {\em walls}. The set of all chambers is denoted by
 $\Ch(\Delta)$. Two chambers are called {\em adjacent} if they contain a common wall. The {\em chamber graph} of $\Delta$ is the graph $G(\Delta)$ with vertex set $\Ch(\Delta)$ where two chambers of $\Delta$ are connected by an edge if they are adjacent. Paths in $G(\Delta)$ are 
 sometimes called {\em galleries} in $\Delta$. In the literature, chamber graphs have also been called dual graphs. This is to distinguish from the graph of a simplicial complex, by which is 
 usually meant its $1$-skeleton.

A pure $d$-dimensional simplicial complex $\Delta$ is said to be {\em balanced} if one can color the vertices of $\Delta$ with the colors $1,2,\ldots,d+1$ so that in every chamber of $\Delta$ the $d+1$ 
vertices are colored differently. Balanced simplicial complexes are a class of complexes with many interesting properties, see e.g. \cite[Ch. III.4]{Stanley1996}. Instead of coloring the vertices of a balanced simplicial complexes $\Delta$, we can think of coloring  the walls of $\Delta$  with $1,2,\ldots,d+1$ so that for every chamber of $\Delta$ all its walls are colored differently. Here, the color of a wall is
taken to be the one that none of its vertices is colored with. This induces an edge-coloring of the chamber graph of $\Delta$.

\begin{example}
In later sections we prove $k$-connectivity for the chamber graphs of certain balanced simplicial complexes 
whose chamber graphs  are $k$-regular.
There is no such relationship between regularity and connectivity  in general. 
Figure \ref{fig:Example_Regular_Connected} shows 
a balanced $1$-dimensional complex whose chamber graph is $6$-regular but not even $5$-connected,
since it is disconnected by removing the $4$ chambers connecting the walls $A$ and $B$
to the left substructure.
\begin{figure}[h]
\begin{center}
\psfrag{A}{$A$}
\psfrag{B}{$B$}
\includegraphics[scale=0.5]{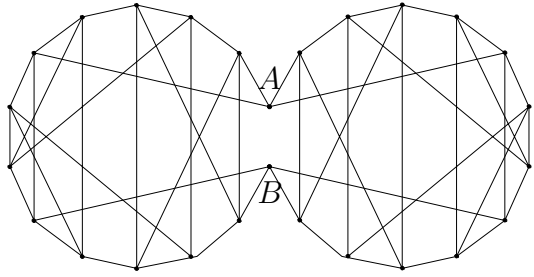}
\caption{Balanced simplicial complex with 6-regular chamber graph that is not 5-connected}
\label{fig:Example_Regular_Connected}
\end{center}
\end{figure}
\end{example}

\section{Chamber graphs of Coxeter complexes}

Let $(W,S)$ be a Coxeter group with a finite set of generators $S$.
The {\em Coxeter complex} $\Delta = \Delta(W,S)$ is by definition a simplicial complex on the vertex set $V = \cup_{s \in S} W / W_{(s)}$ of all left cosets of maximal standard parabolic subgroups $W_{(s)} = W_{S \setminus s}$ of $W$. Its  chambers are  all $C_w = \{ w W_{(s)} : s \in S \}$ for $w\in W$.

We recall some basic facts about $\Delta$, more details can be found in, for example,
\cite{AbramenkoBrown2008}, 
\cite{Ronan1989},
\cite{Scharlau1995},
\cite{Tits1974}.
$\Delta(W,S)$ is an $(|S|-1)$-dimensional balanced simplicial complex. The disjoint union
$V = \cup_{s \in S} V_s$ partitions its vertex set into color classes 
$V_s = W / W_{(s)}$. The group $W$ acts simply transitively 
on the chambers of $\Delta$, which yields a bijection $w \mapsto C_w$ between $W$ and $\Ch(\Delta)$.

Two chambers $C_w$ and $C_{w'}$ are adjacent if and only if $w' = ws$ for some $s \in S$.
Thus, the chamber graph $G = G(\Delta)$ is isomorphic to the Cayley graph of
$W$ with respect to the generating set $S$. In particular, we can take $W$ as vertex set of $G$ and there is an edge between $w$ and $w'$ if and only if $w' = ws$ for some $s \in S$. It is clear that $G$ is $|S|$-regular.

\begin{definition}
A Coxeter group $(W,S)$ is said to be $2$-finite, 
if for each pair $s,t \in S$ the element $s t \in W$ is of finite order.
\end{definition}

Let $(W,S)$ be $2$-finite, and 
for each pair $s,t \in S$ denote by $P_{s,t}$ the path from $s$ to $t$ in $G$ given by
\[
    P_{s,t}: \quad s - st - sts - stst - \ldots - tst - ts - t
\]
The path exists due to the finite order of $s t$, its length is $2k-2$ if the order of $s t$ is $k$.
Observe that if $\{s,t\} \neq \{s', t'\}$, then the paths $P_{s,t}$ and $P_{s' ,t'}$ are disjoint, except
possibly at their endpoints.

\begin{theorem}
\label{connectivitycoxetercomplex}
Let $(W,S)$ be a $2$-finite Coxeter group, and let $\Delta = \Delta(W,S)$ be its Coxeter complex. Then the chamber graph $G = G(\Delta)$  is $|S|$-connected.
\end{theorem}

\begin{proof}
We use Liu's criterion (Lemma \ref{liu}). Clearly, $G$ has at least $|S|+1$ vertices. Let $w$ and $w'$ be vertices with $d_G(w,w') = 2$ and let $w - w'' - w'$ be a path in $G$. Without loss of generality we may 
assume that $w''$ is the identity element $e$, because the action of 
$W$ is vertex-transitive. Then, $w = s$ and $w' = t$ for some $s,t \in S$.

Assume that $|S| = r$. We have the two $s - t$ paths $P_{s,t}$ and $s - e - t$. Furthermore, for every $s' \in S \setminus \{s,t\}$, we can concatenate the paths $P_{s,s'}$ and $P_{s',t}$ to get $r-2$ more $s - t$ paths. It is clear by our above remark that this yields a family of $r$ disjoint $s - t$ paths in $G$.
\end{proof}

\begin{remark}
For the finite case the theorem can be proved also in the following way.
Every finite Coxeter complex $\Delta=\Delta(W,S)$ is  a triangulation of the $(|S|-1)$-sphere
which can be realized as the boundary complex of some $|S|$-dimensional simplicial polytope $P_\Delta$. The chamber graph of $\Delta$ is 
therefore isomorphic to the graph consisting of the vertices and edges of the polytope that is dual to $P_\Delta$.
In \cite{Balinski1961}, M. Balinski showed that the graph of any $d$-dimensional 
convex polytope is $d$-connected. Thus, it follows from these known facts 
that the chamber graph of every finite Coxeter complex $\Delta$ is $|S|$-connected. 

Our proof for the connectivity of chamber graphs of Coxeter complexes, which
explicitly uses the Coxeter group structure of $\Delta$, has two advantages: the argument is valid 
also for many infinite Coxeter complexes, including all affine groups (except $\widetilde{A}_1$)
and hyperbolic groups,
and the construction reappears in a more general form for buildings 
in the next section.

\end{remark}

\begin{example}
\label{S3}

Let $W = \S_4$ be the symmetric group of all permutations on $4$ elements, generated by the set $S = \{ s_1, s_2, s_3 \}$, where $s_i = (i,i+1)$ denotes the adjacent transposition that exchanges the elements $i$ and $i+1$.
Then $(W,S)$  is a Coxeter group.
\begin{figure}[ht]
\begin{center}
\includegraphics[scale=0.5]{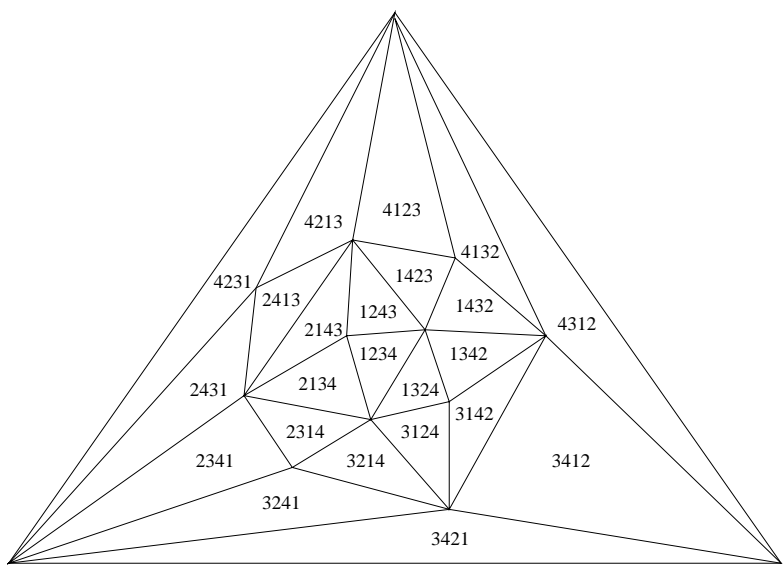}
\caption{Schlegel diagram of the Coxeter complex of $\S_4$}
\label{fig:Sn_schlegel}
\end{center}
\end{figure}

The Coxeter complex $\Delta$ of $(W,S)$ is the barycentric subdivision of the boundary of the $3$-simplex, it triangulates the $2$-sphere. Every chamber of $\Delta$ corresponds to a permutation in $\S_4$. Figure \ref{fig:Sn_schlegel} shows the Schlegel diagram of $\Delta$ where the complex is projected onto the chamber $4321$. Every chamber has been labelled by its permutation written in one-line notation.

\begin{figure}[ht]
\centering
\psfrag{s1}{$s_1$}
\psfrag{s2}{$s_2$}
\psfrag{s3}{$s_3$}
\hfill
\includegraphics[width=0.45\textwidth]{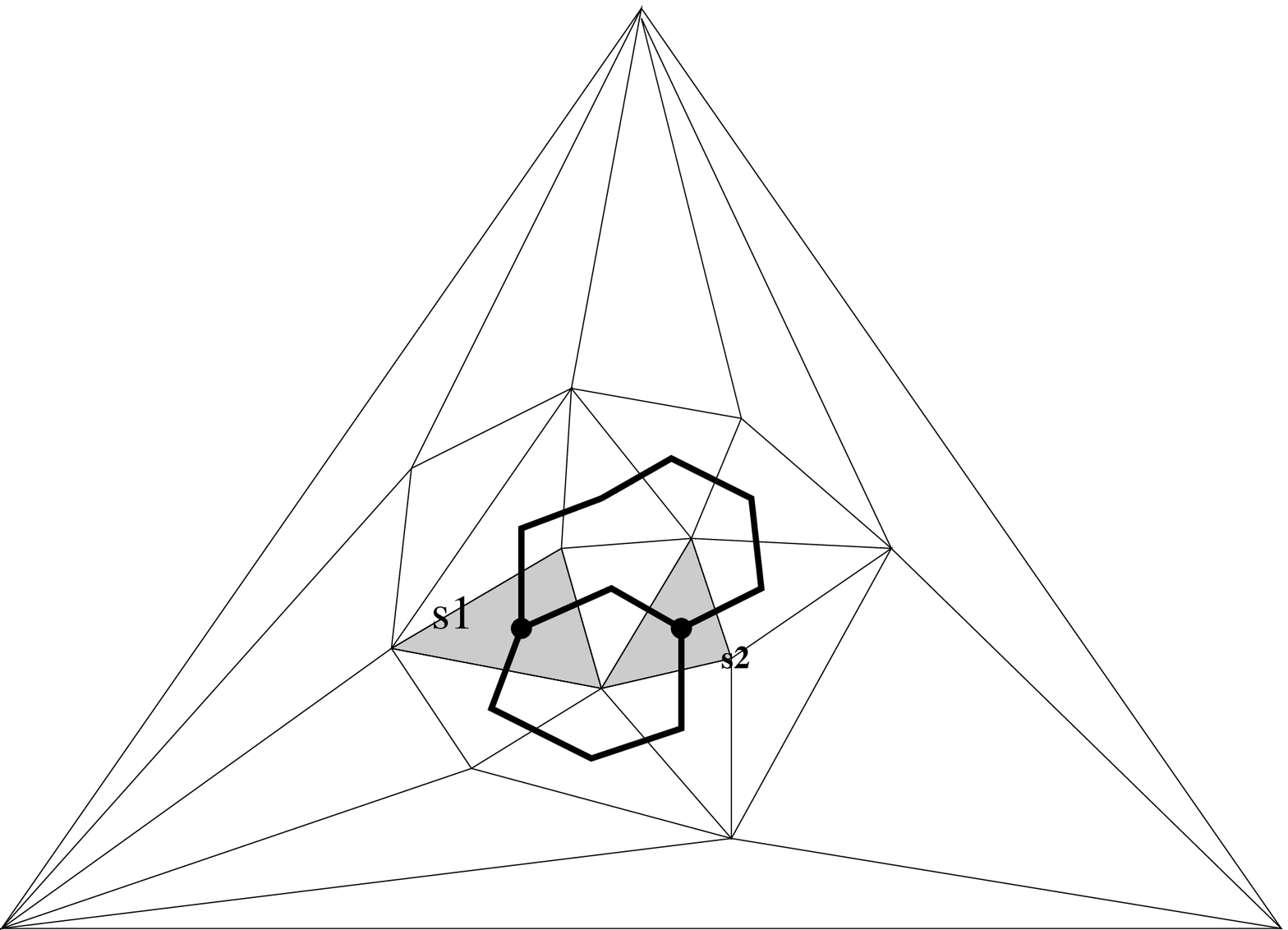}
\hfill
\includegraphics[width=0.45\textwidth]{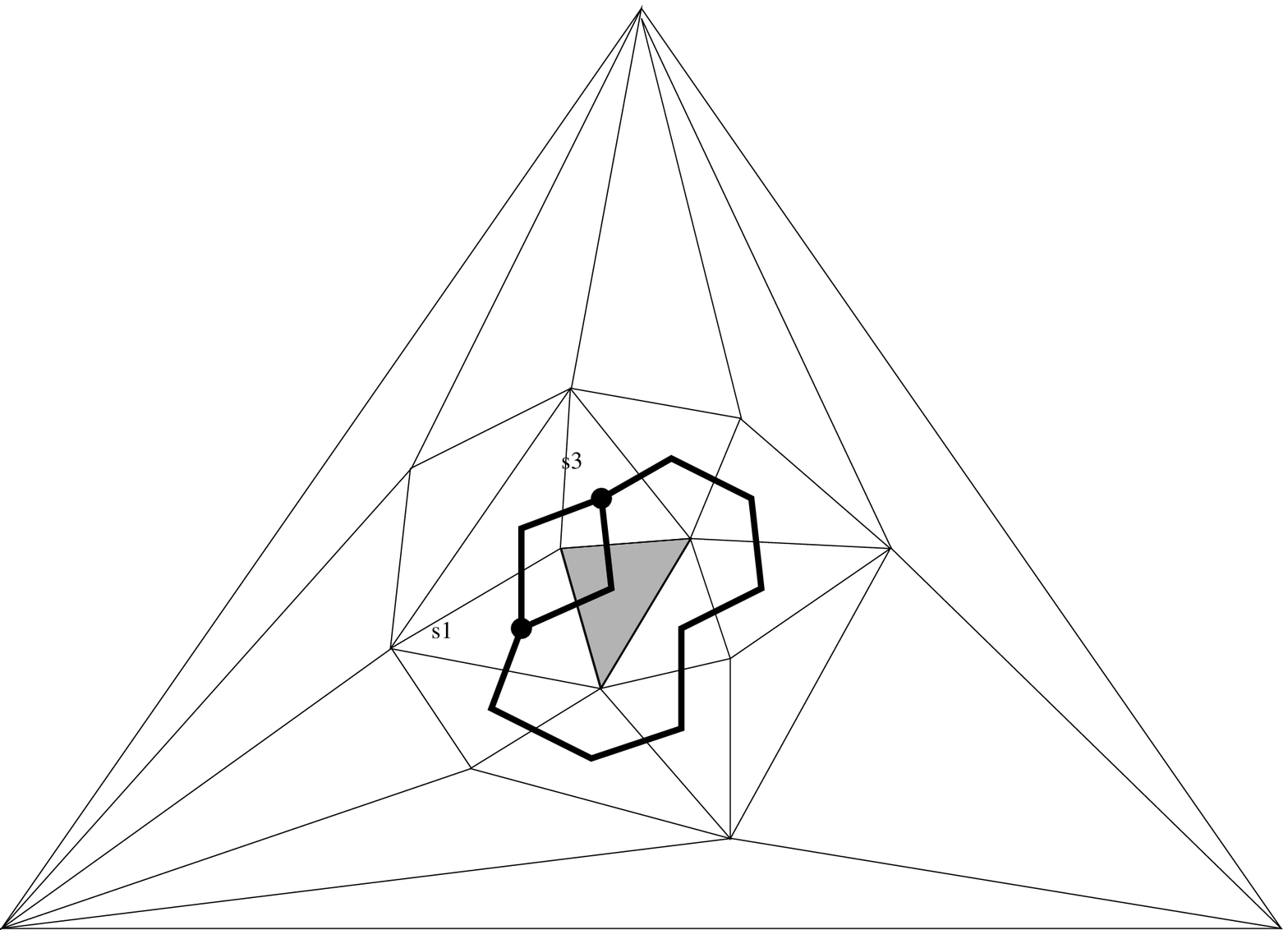}
\hfill
\hfill
\caption{Three disjoint $s_1 - s_2$ resp. $s_1 - s_3$ paths}
\label{fig:Sn_paths}
\end{figure}

Consider the chambers $s_1 = 2134$, $s_2 = 1324$ and $s_3 = 1243$,
 which are all adjacend to the chamber $e = 1234$. 
As constructed in the proof of Theorem \ref{connectivitycoxetercomplex}, the chamber graph of $\Delta$ contains the three disjoint $s_1 - s_2$ paths $s_1 - e - s_2$, $P_{s_1,s_2}$ and $P_{s_1,s_3} \circ P_{s_3, s_2}$, and also the three disjoint $s_1 - s_3$ paths $s_1 - e - s_3$, $P_{s_1,s_3}$ and $P_{s_1,s_2} \circ P_{s_2, s_3}$.
The corresponding  galleries  in the Coxeter complex are indicated  in Figure \ref{fig:Sn_paths}.

\end{example}

\begin{remark}

Note that the condition that the order of $s t$ is  finite for all $s,t \in S$, required in Theorem \ref{connectivitycoxetercomplex}, is necessary, as shown by the following example.

Let $(W,S)$ be the infinite dihedral
 Coxeter group generated by $S = \{s,t\}$ such that the order of $s t$ is infinite. Then, the chamber graph $G$ of the Coxeter complex $\Delta$ is an infinite path, and deleting any node from $G$ 
disconnects the graph. In particular, $G$ is not $2$-connected.

\end{remark}

\section{Chamber graphs of spherical buildings}

In this section we investigate the chamber graphs of buildings in order to
determine their degree of connectivity.

\begin{definition}
A {\em building} is a simplicial complex $\Delta$ which is the union of a 
certain family of subcomplexes $\Sigma$, called apartments, satisfying the 
following axioms:
\begin{itemize}
\item[{\bf (B0)}] Each apartment is a Coxeter complex.
\item[{\bf (B1)}] For any two simplices $A,B \in \Delta$, there is an apartment 
$\Sigma$ containing both of them.
\item[{\bf (B2)}] If $\Sigma$ and $\Sigma'$ are two apartments containing $A$ 
and $B$, then there is an isomorphism $\Sigma \rightarrow \Sigma'$ fixing $A$ 
and $B$ pointwise.
\end{itemize}
\end{definition}

It is not easy, without some prior familiarity with the topic,  
to imagine the elaborate  theory that emanates from these innocent-looking  axioms.
Explanations and details can be found in, for example,
\cite{AbramenkoBrown2008}, 
\cite{Ronan1989},
\cite{Scharlau1995}, and
\cite{Tits1974}.

A direct consequence of axiom  (B2) with $A=B=\emptyset$
and axiom (B0) is that all apartments of a building $\Delta$ are 
isomorphic to the Coxeter complex of some particular Coxeter group $(W,S)$.
Buildings with a finite Coxeter group $W$ are called {\em spherical}.
If every wall is contained in finitely many chambers, then the building is
said to be  {\em locally finite}.
In particular, all finite buildings are both spherical and locally finite.
A building is called {\em thick} if every wall is contained in at least three chambers.

A building $\Delta$ is an $(|S[-1)$-dimensional balanced simplicial complex: its vertex set $V$ 
can be colored by the set $S$ of generators of its Coxeter group.
We want to emphasize another interpretation of balanced-ness, already mentioned in
Section \ref{ch.graphs}: We can associate a {\em type} $s \in S$ 
to every wall of $\Delta$ such that every chamber has exactly one wall of type 
$s$ for every $s \in S$. For a chamber $C \in \Ch(\Delta)$ and some $s \in S$, 
we denote by $N(C,s)$ the set of chambers $D$ such that $D \cap C$ is a wall 
of type $s$.

\begin{lemma}[{\cite[Theorem 5.2.10]{Scharlau1995}}]
\label{Q_Property}
Let $\Delta$ be a thick spherical and locally finite building. Then, there exist 
positive integers $(q_s)_{s \in S}$ such that every wall of type $s$ is 
contained in exactly $q_s + 1$ chambers, or equivalently, such that 
$|N(C,s)| = q_s$ for every $C \in \Ch(\Delta)$ and every $s \in S$.
\end{lemma}

This implies that every thick spherical and locally finite building $\Delta$ has 
a $q$-regular chamber graph, where $q = q(\Delta) = \sum_{s \in S} q_s$.

Recall that a Coxeter group $(W,S)$ is said to be $2$-finite if 
$st$ is of finite order for all $s,t \in S$. Being $2$-finite implies 
existence of the paths $P_{s,t}$ in the chamber graph of the Coxeter complex 
of $(W,S)$, as was described in the previous section. We say that a building 
$\Delta$ is $2$-finite if its Coxeter group $(W,S)$ is. This 
includes all spherical buildings.

Now, let $\Delta$ be a $2$-finite building and let $\Sigma$ be 
an apartment of $\Delta$.
There is an isomorphism $\varphi : \Delta(W,S) \rightarrow \Sigma$, and because 
$W$ acts transitively on $\Ch(\Sigma)$ we can choose any chamber of $\Sigma$ to be 
the image of the chamber $e \in W \cong \Ch(\Delta(W,S))$. Assume that this chamber 
has been fixed.
Then, for any $s,t \in S$, the path $P_{s,t}$ in the chamber graph of 
$\Delta(W,S)$ induces a path in the chamber graph of $\Sigma$ and thus in the 
chamber graph of $\Delta$. We denote this path by $P^{\Sigma}_{s,t}$.

A set $\Gamma$ of chambers in a building is said to be {\em convex} if with each 
pair $C, D\in \Gamma$ every shortest path connecting $C$ and $D$ in the chamber 
graph is completely contained in $\Gamma$. Apartments are known to be convex.

\begin{lemma}
\label{path_fan}
Let $\Delta$ be a thick, spherical and locally finite building. Let $B$ be a 
chamber of $\Delta$, $C \in N(B,s)$ and $t \in S$. Then there is a family of 
$q_t$ paths $(P_D)_{D \in N(B,t)}$ in $G = G(\Delta)$, where $P_D$ is a $C-D$ path 
not containing $B$
for every $D \in N(B,t)$, and such that all paths are pairwise disjoint except at
$C$.
\end{lemma}

\begin{proof}
We distinguish between three cases.

Case 1: $s = t$: This case is easy, since $C$ is here adjacent to all 
chambers $D \in N(B,s)\setminus\{C\}$. Thus, the required paths are trivial, consisting of exactly one step
from $C$ to $D$ for every $D \in N(B,s)\setminus\{C\}$, and the path from 
$C \in N(B,s)$ to itself of length $0$. All paths are obviously disjoint except 
in $C$.

Case 2: $s \not= t$, and $s$ and $t$ commute. For every $D \in N(B,t)$, choose an 
apartment $\Sigma_D$ containing $C$ and $D$. Because $\Sigma_D$ is convex, it 
contains $B$ as well and we choose $B$ as the image of $e \in W = \Delta(W,S)$ 
under the isomorphism between $\Delta(W,S)$ and $\Sigma_D$. Then, $C$ is the 
image of $s$ and $D$ is the image of $t$ in $\Sigma_D$ and the path 
$P^{\Sigma_D}_{s,t}$ has length $2$ and goes from $C$ to $D$ in $G$,
avoiding $B$. Let $E$ be 
the chamber corresponding to $s t$ in $\Sigma_D$, that is $P^{\Sigma_D}_{s,t}$ 
is the path $C - E - D$. Then, $E - D - B$ is a shortest path from $E$ to $B$ 
and thus $D$ is contained in every apartment that contains $B$ and $E$. But $D$ 
is not contained in any apartment $\Sigma_{D'}$ for $D' \not= D$ because 
$\Sigma_{D'} \cap N(B,t) = \{ D' \}$. Thus, $E$ is not contained in any 
$\Sigma_{D'}$ for $D' \not= D$ and the path $P^{\Sigma_D}_{s,t}$ is disjoint 
from any other path $P^{\Sigma_{D'}}_{s,t}$ except in $C$.

Case 3: $s$ and $t$ do not commute. By Lemma \ref{Q_Property}, we have that $|N(B,t)| = |N(C,t)| = q_t$. Thus, we can match the elements of $N(B,t)$ with the elements of $N(C,t)$. For every matched pair $(E,D) \in N(C,t) \times N(B,t)$, we choose an apartment $\Sigma_D$ containing $E$ and $D$. Because $E-C-B-D$ is a shortest path from $E$ to $D$ and $\Sigma_D$ is convex, it contains also $C$ and $B$. We choose $B$ as the image of $e \in W$ under the isomorphism between $\Delta(W,S)$ and $\Sigma_D$. Then, $C, E$ and $D$ are the images of $s$, $s t$ and $t$, respectively. The path $P^{\Sigma_D}_{s,t}$ goes from $C$ via $E$ to $D$,
thus avoiding $B$. An argument analogous to case 2 shows that all paths $P^{\Sigma_D}_{s,t}$ are disjoint except in $C$.

In all three cases we get a path from $C$ to $D$ for every $D \in N(B,t)$, and all these paths are pairwise disjoint except in $C$.
\end{proof}

\begin{theorem}
\label{connectivitybuildings}
Let $\Delta$ be a thick, spherical and locally finite building. 
Then  its chamber graph $G = G(\Delta)$ is $q$-connected, where
$q = q(\Delta)$.
\end{theorem}

\begin{proof}
We use Liu's criterion, Lemma \ref{liu}, to show that $G$ is $q$-connected. 
Because $G$ is $q$-regular, it has at least $q+1$ vertices. 
Let $C,D$ be two chambers of $\Delta$ with $d_G(C,D) = 2$
and let $B$ be a chamber that is adjacent to both $C$ and $D$.
We need to construct $q$ paths from $C$ to $D$ that are disjoint except at their endpoints.

Choose some $s \in S$. By Lemma \ref{path_fan}, there are families of paths $(P'_E)_{E \in N(B,s)}$ and $(P''_E)_{E \in N(B,s)}$ such that $P'_E$ goes from $C$ to $E$ and $P''_E$ goes from $D$ to $E$ for every $E \in N(B,s)$. We join these paths pairwise for every $E$ and get a family of paths $(P_E)_{E \in N(B,s)}$, where the path $P_E$ goes from $C$ via $E$ to $D$. This gives us $q_s$ paths from $C$ to $D$.
Taking those families of paths for every $s \in S$, we get in total $q = \sum_{s \in S} q_s$ paths from $C$ to $D$.

It remains to show that the constructed paths are pairwise disjoint except 
at their endpoints $C$ and $D$. Assume that some chamber $F$ is contained in two paths. Then
$F\neq B$ by construction.
Furthermore, $F$ is not adjacent to $B$, because the chambers adjacent to $B$ are by construction $C$ or $D$ or appear in exactly one path from $C$ to $D$. $F$ is the image of some $w \in W$ for every apartment $\Sigma$ containing $F$ and $B$. If $B$ is chosen to be the image of $e \in W$ in all those apartments, then $F$ is the image of some fixed $w \in W$ independently of the choice of $\Sigma$. This means that if $F$ is contained in any path $P^{\Sigma}_{s,t}$, then $w$ is contained in the dihedral subgroup of $W$ generated by $s$ and $t$. In particular, every path that contains $F$ corresponds to the same two generators $s$ and $t$. But then, $F$ is contained in two paths either from $C$ or from $D$ to two chambers in $N(B,s)$ for some $s$. Those paths are disjoint except in $C$ and $D$ by Lemma \ref{path_fan}.
\end{proof}

\begin{remark}
The theorem is not valid for thick and locally finite buildings that are not spherical. As a 
counterexample we may choose the infinite ternary tree, a rank $2$ building  
whose apartments are the embedded
copies of doubly infinite paths. Here $q(\Delta) =4$, but the chamber graph is not even $2$-connected, since 
removal of any chamber disconnects the graph.
\end{remark}

\section{Chamber graphs of geometric lattices}

For basic definitions of partially ordered 
set theory, see \cite{Stanley1997}.
Throughout the section, we will assume that $P$ is a geometric lattice.

The {\em order complex} $\Delta(P)$ of $P$ is the 
simplicial complex on vertex set $P$ having its totally ordered subsets as simplices.
The chambers of $\Delta(P)$ are the maximal chains of $P$.
We are interested in the chamber graph of $\Delta(P)$ which we will denote by $G(P)$.

Because $P$ is graded, $\Delta(P)$ is a balanced simplicial complex
whose vertices can be labelled by their rank in $P$.
Equivalently, the edges of $G(P)$ are labelled by the rank of the elements of $P$ in which the two incident chambers differ.

Furthermore, $P$ has a minimal element $\hat{0}$ and a maximal element $\hat{1}$. It is an easy observation that $\Delta(P)$ is a double cone over $\Delta(\bar{P})$ where $\bar{P} = P \setminus \{ \hat{0}, \hat{1} \}$. Thus, the chamber graphs $G(P)$ and $G(\bar{P})$ are isomorphic. When we write $G(P)$, we will sometimes mean $G(\bar{P})$ but as the only difference is if the minimal and maximal element are included in the maximal chains, this should not cause any confusion.

Recall that the width of a poset is the maximal size of an antichain. The minimal width of all intervals of length two in $P$ can be considered as {\em local width} of $P$. Define $q(P)$ to be that local width minus one. Equivalently, $q(P)$ is the largest integer such that every open interval of length two in $P$ contains at least $q(P) + 1$ elements.

If $P$ is geometric then every open interval of length two contains at least 2 elements and thus $q(P) \geq 1$. The following Lemma shows that we only need to consider intervals with lower bound $\hat{0}$ in order to compute $q(P)$.

\begin{lemma}
Let $P$ be a geometric lattice, then the following holds.
\[
    q(P) +1 = \min \ \{ \ |(\hat{0},x)| : r(x) = 2 \ \}
\]
\end{lemma}

\begin{proof}
Let $(x,y)$ be an open interval of length two, that is $r(y) = r(x) + 2$. Because $P$ is geometric, we find atoms $a$ and $b$ such that $x \lhd x \vee a \lhd x \vee a \vee b = y$. We claim that $(x,y)$ contains at least as many elements as $(\hat{0},a \vee b)$. That would imply the Lemma.

For every $c \in (\hat{0}, a \vee b)$, the fact that $P$ is a lattice and semi-modular ensures 
that $x \lhd x \vee c \lhd y$. Let $c,c' \in (\hat{0},a\vee b)$ be such that $x \vee c = x \vee c'$. Then $x \vee (c \vee c') = x \vee c \lhd x \vee a \vee b = y$ which implies that $c \vee c' \lhd a \vee b$. But now, $c$, $c'$ and  $c \vee c'$ are atoms and we find that $c = c'$. This shows that the map $(\hat{0},a \vee b) \rightarrow (x,y)$ that sends an atom $c$ to $x \vee c$ is injective, and our claim follows.
\end{proof}

\begin{remark}
Consider a finite building $\Delta$ whose Coxeter group 
is the symmetric group. Equivalently, $\Delta$ is the flag complex of some finite projective geometry \cite[Theorem 6.3]{Tits1974}, that is $\Delta$ is the order complex of the modular and geometric lattice
given by all non-trivial linear subspaces of that geometry, partially ordered by inclusion.

In particular, the Coxeter complex of the symmetric group $S_n$ is the order 
complex of a geometric lattice of rank $n$, namely the Boolean lattice.
\end{remark}

We are interested in the connectivity of the chamber graphs $G(P)$ of geometric 
lattices $P$. 

\begin{theorem}
\label{Lattice}
Let $P$ be a finite geometric lattic of rank $n$ and let $q = q(P)$. Then, the chamber 
graph of $P$ is $q(n-1)$-connected.
\end{theorem}

\begin{proof}
Again, we use Liu's criterion. Clearly, $G(P)$ has at least $q(n-1) + 1$ vertices. Let
$C = (\hat{0} \lhd x_1 \lhd x_2 \lhd \ldots \lhd x_{n-1} \lhd \hat{1})$ and 
$D = (\hat{0} \lhd y_1 \lhd y_2 \lhd \ldots \lhd y_{n-1} \lhd \hat{1})$ be
two chambers of $\Delta(P)$ at distance 2 in $G(P)$. This means
that $x_i = y_i$ for all ranks $i$ except two, say $i_1$ and $i_2$ where we assume that $i_1 < i_2$, and that the chamber adjacent to $C$ through a wall of rank $i_1$ is also adjacent to $D$ through a wall of rank $i_2$. The latter plays a role only if $i_2 = i_1 + 1$ and can be satisfied by exchanging the roles of $C$ and $D$ if necessary.

For every rank $r=1,\ldots,n-1$ we will construct $q$ paths from $C$ to $D$ whose first edge is labeled by $r$, that is paths that first change the element of rank $r$ in the maximal chain $\hat{0} \lhd x_1 \lhd x_2 \lhd \ldots \lhd x_{n-1} \lhd \hat{1}$. For that, we will, without further mentioning, make extensive use of the fact that every open interval of length two contains at least $q+1$ elements.

Case 1: Assume that $|r - i_1| \geq 2$ and $|r - i_2| \geq 2$. Choose $z_1,\ldots,z_q \in (x_{r-1},x_{r+1})$ with $z_j \not= x_r$ for all $j=1,\ldots,q$. Then, for every $j$, there is a path from $C$ to $D$ whose edges are labelled with $r$, $i_1$, $i_2$ and $r$ in that order. The path is shown in Figure \ref{fig:Galleries_Case_1}, where we depicted the chambers as maximal chains and labeled the edges between the chambers. Note that the paths constructed for different $j$ are disjoint except in their endpoints: each interior chamber in the $j$-th path contains the element $z_j$.

Case 2: Assume that $r = i_1 - 1$. Choose $z_1,\ldots,z_q \in (x_{r-1}, x_{r+1})$ with $z_j \not= x_r$ for all $j$ and $w_1,\ldots,w_q \in (x_{r-1}, y_{r+1})$ with $w_j \not= x_r$ for all $j$. Then, $z_j \vee w_j$ covers $z_j$ and $w_j$ and is covered by $x_{r+2}$ because $P$ is semi-modular. For every $j$, there is a path from $C$ to $D$ with edges labeled with $r$, $i_1$, $r$, $i_1$, $i_2$ and $r$ in that order, see Figure \ref{fig:Galleries_Case_2}. Again, those $q$ different paths are internally disjoint, because each interior chamber of the $j$-th path contains either $z_j$ or $w_j$.

Case 3: Assume that $r = i_1 < i_2 - 1$. Let $z_1,\ldots,z_q \in (x_{r-1},x_{r+1})$, we can choose $z_1 = y_r$.
For $j = 1$, we get a path of length $2$ from $C$ to $D$ with edges labeled with $r$ and $i_2$.
For every $j>1$, we get a path from $C$ to $D$ with edges labeled with $r$, $i_2$, $r$ as in Figure \ref{fig:Galleries_Case_3}.
As before, all $q$ paths are internally disjoint.

Case 4: Assume that $r = i_1 = i_2 - 1$. Let $z_1,\ldots,z_q \in (x_{r-1},x_{r+1})$, we can choose $z_1 = y_r$.
For $j = 1$, we get a path of length $2$ from $C$ to $D$ with edges labeled with $r$ and $i_2$ as in Case 3.
For every $j>1$, we get a path from $C$ to $D$ with edges labeled with $r$, $i_2$, $r$, $i_2$ as in Figure \ref{fig:Galleries_Case_4}.
Note that every interior chamber of the $j$-th path contains $z_j$ or $z_j \vee y_r$ and that $z_j \vee y_r \not= z_{j'} \vee y_r$ for $j \not= j'$ because $P$ is a lattice. This ensures that all $q$ paths are internally disjoint.

Case 5: $i_1 + 1 = r = i_2 - 1$. Choose $z_1,\ldots,z_q \in (x_{i_1},x_{i_2})$ with $z_j \not= x_r$ for all $j$, $w_1,\ldots,w_q \in (y_{i_1},y_{i_2})$ with $w_j \not= x_r$ for all $j$ and for every $j=1,\ldots,q$, choose $u_j \in (x_{i_1-1}, w_j)$ with $u_j \not= y_{i_1}$. Then for every $j$, we find a path from $C$ to $D$ as shown in Figure \ref{fig:Galleries_Case_5}. Furthermore, the elements $u_j \vee x_{i_1}$ and $u_j \vee z$ are different for different $j$ and the constructed $q$ paths are internally disjoint.

Case 6: $r = i_2 = i_1 + 1$. Choose $w_1,\ldots,w_q \in (x_{i_1-1}, y_{i_2})$ with $w_j \not= x_{i_1}$ for all $j$. For every $j$, we construct a path from $C$ to $D$ as in Figure \ref{fig:Galleries_Case_6}. Every interior chamber of the $j$-th path contains $w_j$ or $x_{i_1} \vee w_j$ and $w_j \vee x_{i_1} \not= w_{j'} \vee x_{i_1}$ for $j \not= j'$ because $P$ is a lattice. Thus, all $q$ paths are internally disjoint.

Cases 7--8--9: Assume that $r = i_1 + 1 < i_2 - 1$, or $i_1 + 1 < r = i_2 - 1$, or $r = i_2 + 1$.
Then, by a construction similar to that used for Case 2 
we can construct $q$ paths from $C$ to $D$ that are disjoint except in their endpoints

Case 10: Assume that $r = i_2 > i_1 + 1$.
Then, a construction similar to that used for Case 3 yields $q$ paths
 from $C$ to $D$ that are disjoint except in their endpoints. \medskip

We invite the reader to check that every path from $C$ to $D$ that we constructed above and that starts with an edge labeled by $r$ does not change any elements in the maximal chains except at the ranks $i_1, i_2$ and $r$. Furthermore, every maximal chain different from $C$ and $D$ contains some element at rank $r$, $i_1$ or $i_2$ that is not contained in any maximal chain of any other path. This shows that we have
constructed a family of $q(n-1)$ paths from $C$ to $D$ that are pairwise disjoint except in their endpoints. Liu's criterion implies that $G(P)$ is $q(n-1)$-connected.
\end{proof}

\begin{example}
Recall that the Coxeter complex of the symmetric group $S_4$ is isomorphic to 
the order complex of the boolean lattice $B_4$ consisting of all subsets of $\{1,2,3, 4\}$. 
We have shown in Theorem \ref{connectivitycoxetercomplex} that the chamber 
graph of that complex is $3$-connected.
This result also follows from Theorem \ref{Lattice} because $q(B_4) = 1$. In fact, the construction in the proof of Theorem \ref{Lattice} yields exaclty the same family of internally disjoint paths as in Example \ref{S3}.
\end{example}

\begin{figure}[htb]
\begin{center}
\psfrag{C}{$C$}
\psfrag{D}{$D$}
\psfrag{xr}{$x_r$}
\psfrag{xi1}{$x_{i_1}$}
\psfrag{xi2}{$x_{i_2}$}
\psfrag{yi1}{$y_{i_1}$}
\psfrag{yi2}{$y_{i_2}$}
\psfrag{r}{$r$}
\psfrag{i1}{$i_1$}
\psfrag{i2}{$i_2$}
\psfrag{z}{$z$}
\includegraphics[scale=0.45]{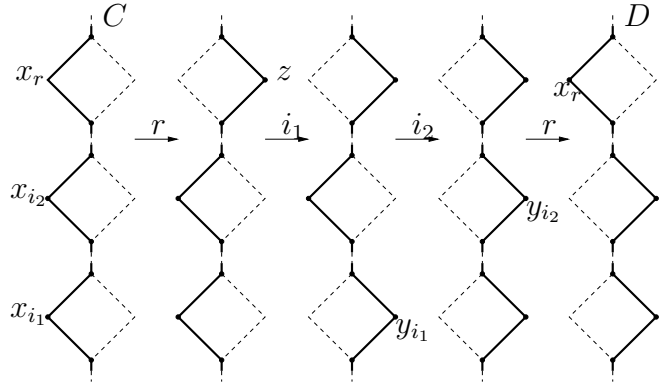}
\caption{Path from $C$ to $D$, Case 1}
\label{fig:Galleries_Case_1}
\end{center}
\end{figure}

\begin{figure}[htb]
\begin{center}
\psfrag{C}{$C$}
\psfrag{D}{$D$}
\psfrag{xr}{$x_r$}
\psfrag{xi1}{$x_{i_1}$}
\psfrag{xi2}{$x_{i_2}$}
\psfrag{yi1}{$y_{i_1}$}
\psfrag{yi2}{$y_{i_2}$}
\psfrag{r}{$r$}
\psfrag{i1}{$i_1$}
\psfrag{i2}{$i_2$}
\psfrag{z}{$z$}
\psfrag{w}{$w$}
\psfrag{zvw}{$z\! \vee\! w$}
\includegraphics[scale=0.45]{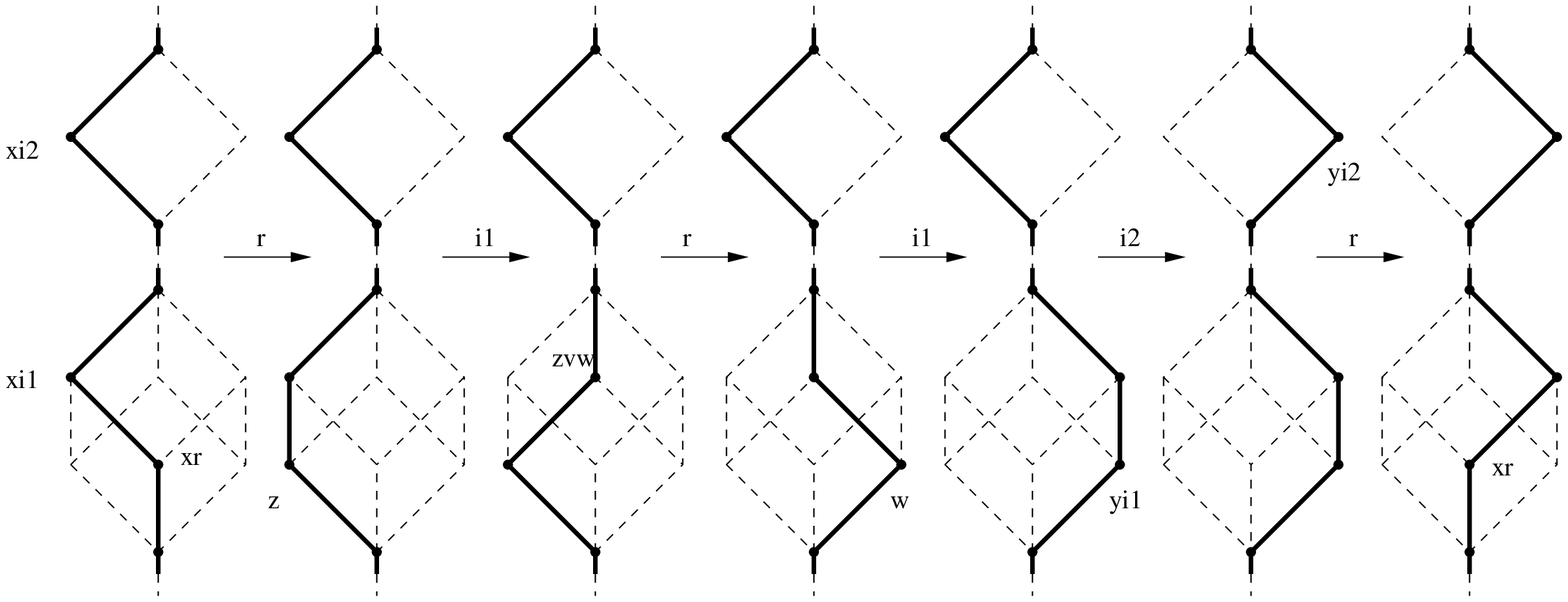}
\caption{Path from $C$ to $D$, Case 2}
\label{fig:Galleries_Case_2}
\end{center}
\end{figure}

\begin{figure}[htb]
\begin{center}
\psfrag{C}{$C$}
\psfrag{D}{$D$}
\psfrag{xi1}{$x_{i_1}$}
\psfrag{xi2}{$x_{i_2}$}
\psfrag{yi1}{$y_{i_1}$}
\psfrag{yi2}{$y_{i_2}$}
\psfrag{i1}{$i_1$}
\psfrag{i2}{$i_2$}
\includegraphics[scale=0.45]{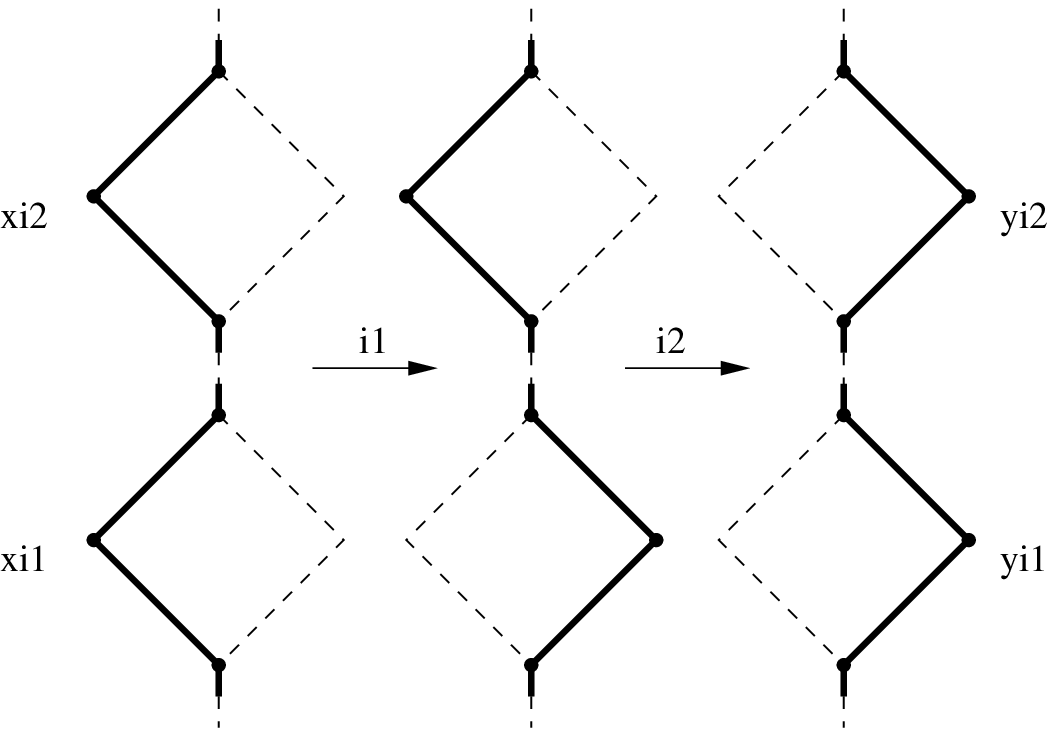}
\caption{Path from $C$ to $D$, Case 3}
\label{fig:Galleries_Case_3}
\end{center}
\end{figure}

\begin{figure}[htb]
\begin{center}
\psfrag{C}{$C$}
\psfrag{D}{$D$}
\psfrag{xr}{$x_r$}
\psfrag{xi1}{$x_{i_1}$}
\psfrag{xi2}{$x_{i_2}$}
\psfrag{yi1}{$y_{i_1}$}
\psfrag{yi2}{$y_{i_2}$}
\psfrag{r}{$r$}
\psfrag{i1}{$i_1$}
\psfrag{i2}{$i_2$}
\psfrag{z}{$z$}
\psfrag{w}{$w$}
\psfrag{u}{$u$}
\psfrag{zvu}{$z\! \vee\! u$}
\psfrag{xi1vu}{$x_{i_1}\!\! \vee\! u$}
\includegraphics[scale=0.45]{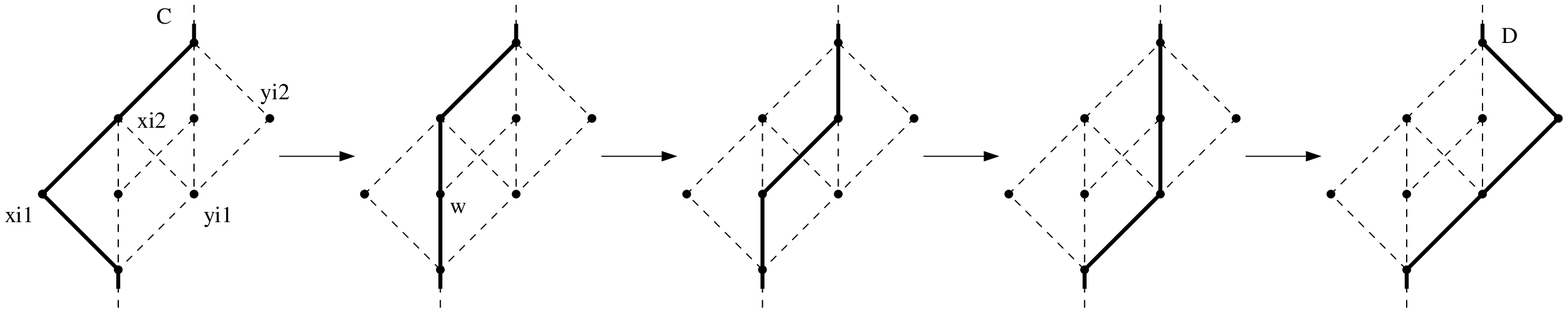}
\caption{Path from $C$ to $D$, Case 4}
\label{fig:Galleries_Case_4}
\end{center}
\end{figure}

\begin{figure}[htb]
\begin{center}
\psfrag{C}{$C$}
\psfrag{D}{$D$}
\psfrag{xr}{$x_r$}
\psfrag{xi1}{$x_{i_1}$}
\psfrag{xi2}{$x_{i_2}$}
\psfrag{yi1}{$y_{i_1}$}
\psfrag{yi2}{$y_{i_2}$}
\psfrag{r}{$r$}
\psfrag{i1}{$i_1$}
\psfrag{i2}{$i_2$}
\psfrag{z}{$z$}
\psfrag{w}{$w$}
\psfrag{u}{$u$}
\psfrag{zvu}{$z\! \vee\! u$}
\psfrag{xi1vu}{$x_{i_1}\!\! \vee\! u$}
\includegraphics[scale=0.45]{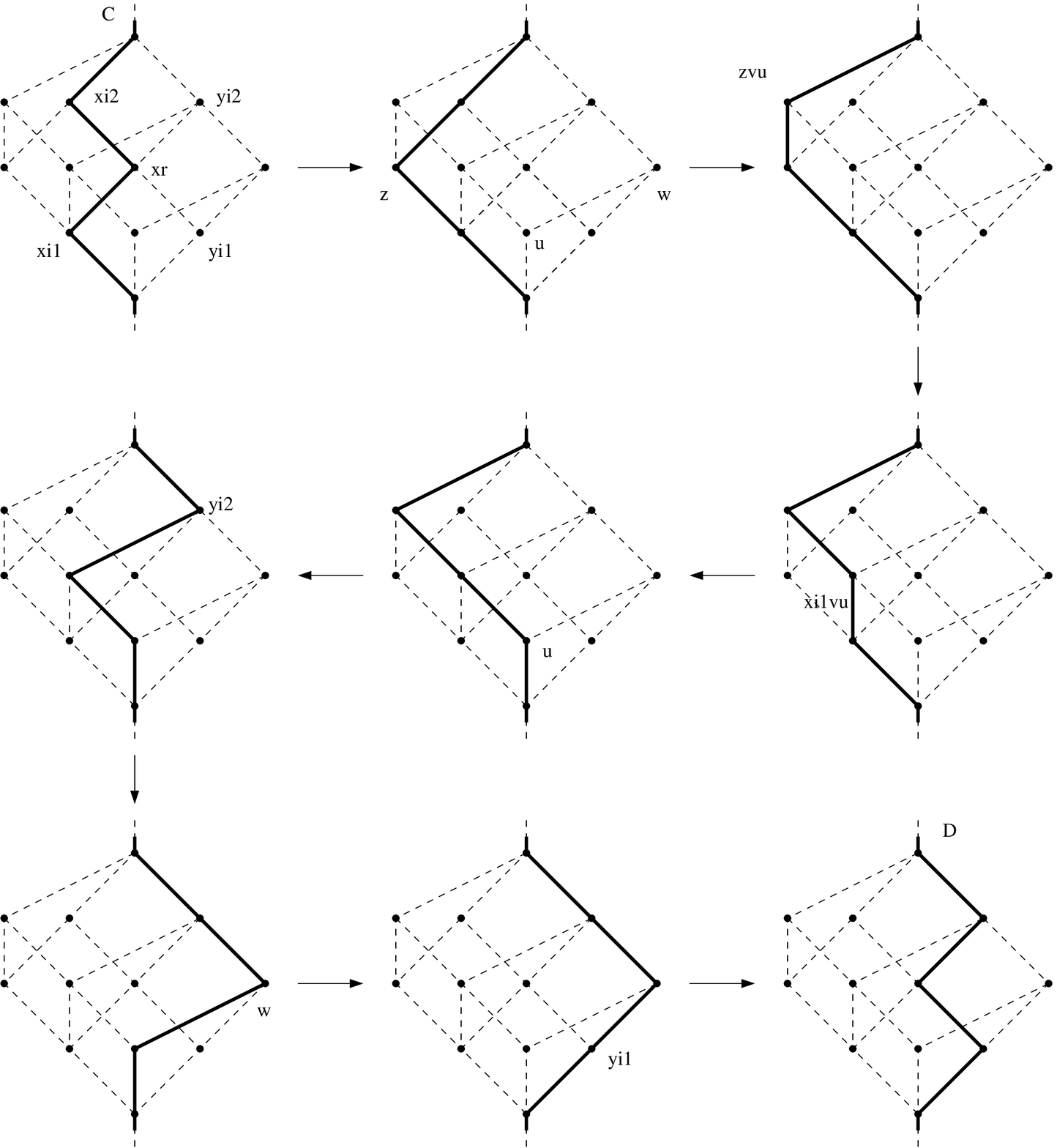}
\caption{Path from $C$ to $D$, Case 5}
\label{fig:Galleries_Case_5}
\end{center}
\end{figure}

\begin{figure}[htb]
\begin{center}
\psfrag{C}{$C$}
\psfrag{D}{$D$}
\psfrag{xi1}{$x_{i_1}$}
\psfrag{xi2}{$x_{i_2}$}
\psfrag{yi1}{$y_{i_1}$}
\psfrag{yi2}{$y_{i_2}$}
\psfrag{i1}{$i_1$}
\psfrag{i2}{$i_2$}
\psfrag{w}{$w$}
\psfrag{xi1vw}{$x_{i_1}\!\! \vee\! w$}
\includegraphics[scale=0.45]{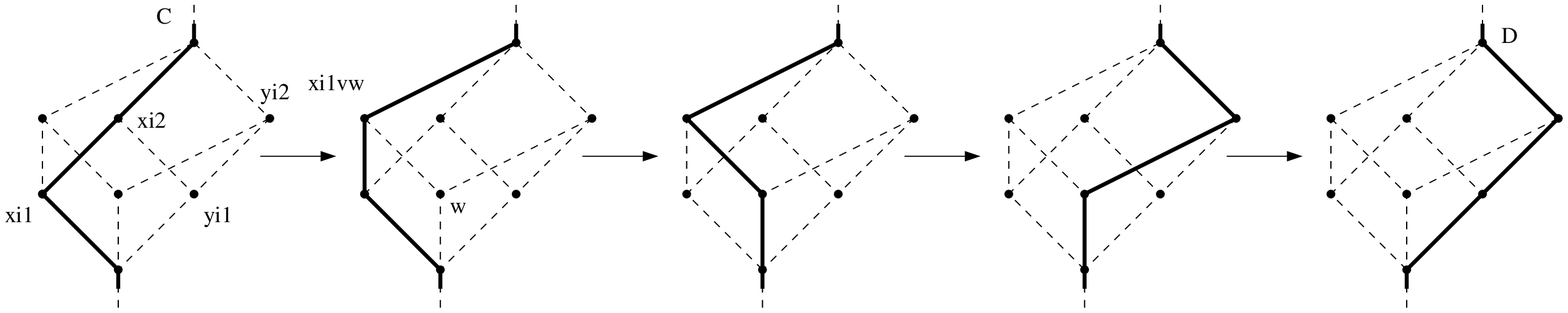}
\caption{Path from $C$ to $D$, Case 6}
\label{fig:Galleries_Case_6}
\end{center}
\end{figure}

\end{document}